\documentclass[11pt, a4paper]{amsart}
\usepackage[utf8]{inputenc}
\usepackage{amsmath}
\usepackage{amsfonts}
\usepackage{amssymb}
\usepackage{graphicx}
\usepackage{amsmath,amssymb, amsthm}
\usepackage{pgf}
\usepackage{tikz-cd}
\usepackage{xy}
\usepackage{hyperref}
\usetikzlibrary{matrix,arrows,decorations.pathmorphing}
\usepackage[english]{babel}
\usepackage{mathtools}
\usepackage{caption}
\usepackage{geometry}
\usepackage{float}

\theoremstyle{definition}

\theoremstyle{definition}
\newtheorem{rmk}{Remark}
\theoremstyle{definition}

\theoremstyle{definition}\usepackage{amsmath}

\theoremstyle{plain}
\newtheorem{thm}{Theorem}
\theoremstyle{plain}
\newtheorem{prop}{Proposition}
\theoremstyle{plain}

\theoremstyle{plain}
\newtheorem{lem}{Lemma}

\newtheorem{coro}{Corollary}

\newcommand{\supp}{\operatorname{supp}}

\newcommand{\R}{\mathbb{R}}

\newcommand{\lp}{\left(}
\newcommand{\rp}{\right)}
\newcommand{\N}{\mathbb{N}}

\newcommand{\spa}{\operatorname{span}}
\newcommand{\esup}{\operatorname{ess\,sup}}

\title{A note on the Moment Problem for codimension greater than $1$}
\author{Francesco Battistoni, Enrico Miglierina}
\address{Dipartimento di Matematica per le Scienze Economiche, Finanziarie ed Attuariali\\
	Universit\`{a} Cattolica\\
	via Necchi 9\\
	20123 Milano\\
	Italy}
\email{francesco.battistoni@unicatt.it}

\email{enrico.miglierina@unicatt.it}
\date{}

\makeatletter
\@namedef{subjclassname@2020}{%
	\textup{2020} Mathematics Subject Classification}
\makeatother



\begin{document}
	
	\begin{abstract}
		We provide new conditions under which the alternating projection sequence  converges in norm for the convex feasibility problem where a linear subspace with finite codimension $N\geq 2$ and a lattice cone in a Hilbert space are considered. The first result holds for any Hilbert lattice, assuming that the orthogonal of the linear subspace 
		admits a basis made by disjoint vectors with respect to the lattice structure. The second result is specific for $\ell^2(\N)$ and is proved when only one vector of the basis is not in the cone but the sign of its components is definitively constant and its support has finite intersection with the supports of the remaining vectors.
	\end{abstract}
	\maketitle
	
	\section{Introduction}
	Given a Hilbert space $X$ and two closed convex nonempty subsets $A$ and $B$, it is a typical problem to study the intersection $A\cap B$ and, in particular, to detect which points belong to this intersection. This is known as the  {\em convex feasibility problem}, and it appears through several fields of mathematics, both in theoretical studies and in applications. In fact, many different problems can be reformulated as convex feasibility problems: as examples, we mention approximation theory \cite{Deutsch92}, solution of convex inequalities, partial differential equations, minimization of convex non-smooth functions on the theoretical side (see, e.g., the monographs \cite{BorweinZhu05book,BauschkeCombettes17book} and the references therein), whereas we recall image reconstruction (\cite{Combettes1996,CensorGibaliLenzenSchnorr16,ZhaoKobis18}, planning of medical therapies \cite{CensorAltschulerPowlis88}, signal processing \cite{CarmiCensorGurfil12} sensor network location problem \cite{GholamiTetraushviliStromCensor13} on the applications side.
	
	Researchers investigated convex feasibility by studying whether specific sequences converge in norm to points of $A\cap B$ (or points which minimize the distance between $A$ and $B$, in case the intersection is empty). Many of these methods are settled in Hilbert spaces and they are based on projections; a valuable source for an overview on projection methods is \cite{CensorCegielskiBibliography}. The most used sequence of this kind is the so called {\em alternating projection sequence}, which starting from a point $b_0\in X$ consists in the two sequences of the form $a_{n}\coloneqq P_A(b_{n-1})$ and $b_{n}\coloneqq P_B(a_n)$, where $P_A$ and $P_B$ are the projections on $A$ and $B$ respectively.  This study was started by Von Neumann \cite{VonNeumann49} assuming that $A$ and $B$ are linear subspaces of $X$: in this case the norm-convergence is assured from any starting point $b_0$, and for every point $x\in A\cap B$ a suitable alternating sequence converging to $X$ can be produced. When considering more general convex sets $A$ and $B$ with $A\cap B\neq \emptyset$, Bregman \cite{Bregman65} proved that weak convergence to a point in $A\cap B$ always holds; however, it was only some decades later that Hundal \cite{Hundal04} and Kopecka \cite{Kopecka08} provided examples where norm convergence failed. These results further justified the research, started in previous years, for conditions to be put on $A$ and $B$ in order to guarantee the norm convergence. Many results of this kind were found by Bauschke and Borwein \cite{BauschkeBorwein93}, who in particular linked the problem to the concept of regularity of the couple $(A,B)$. This was later reprised by De Bernardi, Molho and Miglierina \cite{DeBernardiMiglierinaMolho, DeBernardiMiglierina21, DeBernardiMiglierina22}  giving characterizations for generalizations of the alternating projection sequence. In fact, we also mention that several others projection sequences and methods were studied for the convex feasibility problem: some overviews may be found in \cite{BauschkeBorwein96, BorweinSimsTam15, HenrionMalick11}.\\\\
	A problem often arising in applications is to find the solutions of $T(x)=\bar{y}$ under non-negativity constraint, where $T:X\rightarrow \R^N$ is a linear operator from a Hilbert lattice $X$ with lattice cone $X^+$ to $\R^N$. This problem can be seen as a convex feasibility problem where the sets are respectively the affine subspace $T^{-1}(\bar{y})$ and the lattice cone $X^+$, ad it is usually called {\em moment problem}. 
	Bauschke and Borwein \cite[Section 5]{BauschkeBorwein93} proved that the alternating projection sequences with respect to $A=T^{-1}(\bar{y})$ and $B=X^+$ converge in norm if one of the following is satisfied:
	\begin{itemize}
		\item $Q(B)\subseteq B$, where $Q$ is the projection on $A^{\perp}$;
		\item $A \cap B^{\oplus}$, where $B^{\oplus}$ is the polar of the cone $B$;
		\item $N=1$. This case is shown to be a particular instance of the previous two conditions.
	\end{itemize}
	For higher codimension $N\geq 2$, no results concerning the norm convergence seem to be known, and surely were not known when the overview \cite{BorweinSimsTam15} was written.  Indeed, the authors of this paper pointed out that in \cite{BauschkeBorwein93} \textquotedblleft Bauscke and Borwein proved that the method of alternating projection converges in norm whenever the affine subspace has codimension one. The same was conjectured to stay true for any finite codimension, but remains a stubbornly open problem.\textquotedblright
	
	In this paper we focus on a particular instance of moment problem corresponding to finding the non-negative zeroes of the operator $T$, i.e., we assume $\bar{y}=0$. This is the same to consider that $A$ is a linear subspace of $X$ with finite codimension $N$ and $B=X^+$, so that $A\cap B$ contains $0$ at least.
	
	The goal of this paper is to present some new conditions which ensure the norm convergence for the moment problem in higher codimension: in particular, we focus on characterizing the orthornormal basis $\{v_1,v_2,\ldots,v_N\}$ of $A^{\perp}$, 
	and examples will be given in order to show that these characterizations do not fall into the already known conditions from \cite{BauschkeBorwein93} described above. A special emphasis will be given on the case $N=2$, with the results for higher values of $N$ derived by generalizations of this specific case.
	
	The paper is organized as follows.
	In Section 2 a short recall of the main concepts and properties needed for our study is given, including some properties of Hilbert lattices and a brief review of known results about the moment problem.
	
	In Section 3 we prove our main result: norm convergence of alternating projections sequences by assuming that the vectors $v_i$ are pairwise disjoint, both when $X=\ell^2(\N)$ and $X=L^2(Y,\mu)$ for some measure space $(Y,\mu)$. The last result allows us to prove the analogue claim for a generic Hilbert lattice $X$.
	
	In Section 4, we try to move a step further by considering the case where the vectors $v_i$ have pairwise finite intersection for their supports in $\ell^2(\N)$. Under this assumption, we are able to prove the norm convergence of alternating projection sequences only if we add an additional requirement. Namely, we assume that one of them is not in the cone $B$, but the sign of its components is definitively constant.
	\\
	\\
	\textit{Acknowledgements.} The authors wish to thank Aris Daniilidis and Carlo De Bernardi for insightful discussions concerning this problem. The authors are members of the INdAM group GNAMPA. The second author is partially supported by  INdAM - GNAMPA Project” CUP E53C23001670001 and by the Ministry for Science and Innovation, Spanish State Research Agency (Spain),
	under project PID2020-112491GB-I00.

	\section{Preliminary notions}
	\subsection{Support and essential support} Let $X=\ell^2(\N)$ and let $x\in X$: we have $x=(x_1,x_2,\ldots,x_n,\ldots)$ satisfying $\sum_{i} x_i^2 < \infty$. We recall that the norm of $x$ is $\|x\|\coloneqq \lp \sum_{i} x_i^2\rp^{1/2}$.
	
	We represent a sequence of elements in $X$ with the notation $\{a^n\}$ or $\{b^n\}$: this means that for every $n$ we have $a^n = (a^n_1,a^n_2,\ldots,a^n_i,\ldots)$
	. 
	
	The {\em support of $x$} is the set
	$$
	\supp(x)\coloneqq \{n\in\N\colon x_n\neq 0\}.
	$$
	
	\begin{lem}\label{LemmaWeakConvFiniteSupport}
		Let $\{a^n\}$ be a sequence in $X=\ell^2(\N)$ which weakly converges to $x\in X$. Assume there exists a finite set $S\subset \N$ such that $\supp(a^n)\subseteq S$ for every $n$. Then the sequence $\{a^n\}$ converges in norm to $x$.
	\end{lem}
	
	\begin{proof}
		The weak convergence implies that $a^n$ pointwise converges to $x$, i.e. for any fixed $i\in\N$ it is $a^n_i \to x_i$ for $n\to +\infty$. Now, we have $a^n_i=0$ for every $n\in\N$ and for every $i\in\N\setminus S$, so that $\supp(x)\subseteq S$. It follows that 
		$$
		\|a^n-x\|^2\leq \sum_{i\in S} (a^n_i - x_i)^2 \to 0\quad\text{ as }n\to +\infty
		$$
		and this is true since $S$ is a finite set.
	\end{proof}
	A similar notion of support can be given in more general function spaces. Let $(Y,\mu)$ be a measure space and let $f\in L^2(Y,\mu)$: the {\em essential support of $f$} is defined as the set
	$$
	\esup(f)\coloneqq Y\setminus\; \bigcup \,\left\{A\subseteq Y\colon A\text{ open and } f = 0 \;\,\mu\text{-almost everywhere on }A\right\}.
	$$
	It is a simple matter to check that this definition does not depend on the representative for the equivalence class of $f$. 
	Moreover, it reduces to the previous definition of support if we think of $\ell^2(\N)$ as $L^2(\N,\mu_c)$ with $\mu_c$ being the discrete counting measure. 

	\subsection{Hilbert lattices} 
	Let $X$ be a vector space and $C\subset X$ be a {\em cone} (i.e, a set $C$ such that $C+C\subseteq C$, $\alpha C\subseteq C$ for all $\alpha \geq 0$ and, $C \cap (-C)=\{0\}$). Let us define an order structure on $X$: we say $x\leq_C y$ if $y-x \in C$. Then, $X$ is a {\em lattice} if for every $x,y\in X$ there exist $x \land y\coloneqq \inf\{x,y\}$ and $x \lor y\coloneqq \sup\{x,y\}$ where $\inf$ and $\sup$ are the greatest lower bound and the least upper bound  with respect to $\leq_C$, respectively. In this case the cone $C$ in $X$ is called {\em lattice cone} of $X$ and it is denoted by $X^+$. For every $x\in X$, we define the positive part $x^+\coloneqq x\lor 0$, the negative part $x^-\coloneqq (-x)\lor 0$ and the modulus $|x|_{X^+}\coloneqq x^+ + x^-$. 
	
	A {\em normed vector lattice} $X$ is a vector lattice equipped with a {\em lattice norm} (i.e., a norm $\|\cdot\|$ on $X$ such that $|x|_{X^+}\leq_{X^+} |y|_{X^+}$ implies $\|x\|\leq \|y\|$). A Hilbert space $X$, endowed with the norm $\|\cdot\|$, is a {\em Hilbert lattice} if $X$ is a complete normed vector lattice with respect to the original norm $\|\cdot\|$ (see, e.g., \cite{Schaefer74book,Meyer-Nieberg91book}).
	
	Well known Hilbert lattices are the spaces $L^2(Y,\mu)$ and $\ell^2(\N)$ where the lattice cones are the cone $L^2_+(Y,\mu)$ of the functions $\mu$-almost everywhere non-negative on $Y$ and the cone $\ell^2_+(\N)$ of the elements with non-negative coordinates, respectively.
	The following result shows that a Hilbert lattice is always \textquotedblleft equal" to a space $X=L^2(Y,\mu)$ for a suitable choice of the measure $\mu$. 
	\begin{prop}\label{PropositionIsoLattice}
		Let $X$ be a Hilbert lattice. There exist a measure space $(Y,\mu)$ and a Hilbert space isomorphism $f: X \to L^2(Y,\mu)$ which is isometric, sends $X^+$ to $(L^2(Y,\mu))^+$ and preserves the lattice structure. In particular, $f(x\lor y) = f(x)\lor f(y)$ and $f(x\land y)=f(x)\land f(y)$.
	\end{prop}
	\begin{proof}
		This is proved in \cite[Cor. 2.7.5]{Meyer-Nieberg91book} as a consequence of the more general result in \cite[Thm 2.7.1]{Meyer-Nieberg91book} which holds for Banach lattices. 
	\end{proof}
	Now, we collect some properties of Hilbert lattices that will be useful in the sequel of the paper.
	The first result we mention is a result about the existence of the limit for a monotonic sequence, where monotonicity is defined with respect to the lattice order. We point out that this result holds true in a more general framework, namely in each reflexive Banach lattice.   
	\begin{prop}\label{PropositionDecreasing}
		Let $X$ be a Hilbert lattice. Let $\{x^n\}$ be a norm-bounded sequence in $X$ such that $x^{n+1}\leq_{X^+} x^n$ for every $n$. Then the sequence converges in norm.
	\end{prop}
	For the sake of convenience of the reader, we provide a short proof of this proposition.
	\begin{proof}
		Since $X$ is a Hilbert lattice, by combining Lemma 2.39 and Theorem 2.45 in \cite{AliprantisTourkyCones}, we have  that a sequence $\{y^n\}$ satisfying $0\leq_{X^+} y^n$,\; $y^n\leq_{X^+} y^{n+1}$ for every $n$ and $\sup \|y^n\|<\infty$ converges in norm.  Now, we consider the sequence $\{y^n\}$ where $y^n\coloneqq x^1-x^n$. This sequence satisfies the conditions above, and the norm convergence of $\{y^n\}$ implies the one of $\{x^n\}$.
	\end{proof}
	\begin{rmk}
		Let $X$ be a Hilbert lattice, then $P_{X^+}(x)=x^+$ for every $x \in X$, where $P_A(x)$ denotes the standard projection of an element $x$ on the convex set $A$ (see \cite{BauschkeBorwein93}).\\
		\noindent Moreover, we are interested in two particular cases:
		\begin{itemize}
			\item If $X=\ell^2(\N)$ and $X^+$ is the cone formed by elements with non-negative coordinates, then $x^+ = P_{X^+}(x)$ has coordinates equal to the positive part (as real numbers) of the coordinates of $x$;
			\item If $X=L^2(Y,\mu)$ and $X^+$ is the cone whose elements are functions $\mu$-almost everywhere non-negative on $Y$, then $x^+$ is the equivalence class of the function obtained by taking the positive part (as functions with values in $\R$) of $x$. 
		\end{itemize}
	\end{rmk}

	We conclude this subsection by considering disjoint elements in a Hilbert lattice.
	We recall that two elements $x,y$ in a vector lattice $X$ are {\em disjoint} if $|x|_{X^+}\land |y|_{X^+} =0$.
	
	\begin{prop}\label{PropositionDisjointEssSup}
		Let $f_1,f_2 \in L^2(Y,\mu)$. Then $f_1$ and $f_2$ are disjoint with respect to the structure induced by $(L^2(Y,\mu))^+$ if and only if 
		$$\esup(f_1)\cap \esup(f_2)=\emptyset.$$
	\end{prop}
	
	\begin{proof}
		In this lattice structure, the element $f\land g$ is defined as the equivalence class of the function $\inf_{y\in Y}\{f(y),g(y)\}$: this definition does not depend on the representatives for $f$ and $g$.
		
		We can write $\esup(f_i)$ as $Y\setminus A_i$, where $A_i$ is the largest open set of $Y$ such that $f_i$ is $\mu$-almost everywhere null on $A_i$. Therefore the condition $\esup(f_1)\cap \esup(f_2)=\emptyset$ is equivalent to $A_1\cup A_2=Y$.
		
		Assume $f_1$ and $f_2$ are disjoint: then $|f_1|\land |f_2|$ is the class of the null function. If $Y \supsetneq A_1\cup A_2 $, the set $S\coloneqq Y\setminus (A_1\cup A_2)$ cannot have measure zero, otherwise $Y = A_1\cup A_2$ by definition of the sets $A_i$, and so $S$ has positive (possibly infinite) measure. But $S$ is such that neither $f_1$ and $f_2$ are $\mu$-almost everywhere 0 on it: therefore $|f_1|\land |f_2|$ is not the class of the null function when restricted to $S$, which contradicts the hypothesis.
		
		If instead we assume $Y=A_1\cup A_2$, the fact that $f_i = 0$ almost everywhere when restricted to $A_i$ immediately gives $|f_1|\land |f_2|=0$ almost everywhere, i.e. the two functions are disjoint.
	\end{proof}

	Finally, it is worth noting that, by Propositions \ref{PropositionIsoLattice} and \ref{PropositionDisjointEssSup}, disjoint elements in a Hilbert lattice $X$ are orthogonal with respect to the inner product of $X$.

	\subsection{Results on alternating projections} Let $X$ be a Hilbert space and $A,B\subseteq X$ closed convex and non-empty. From now on, we assume $A\cap B\neq \emptyset$.
	
	Given $b^0\in X$, define the {\em alternating projection sequences} as
	\begin{equation}\label{alternatingProjectionDefinition}
		a^{n+1}\coloneqq P_A(b^n),\quad b^{n+1}\coloneqq P_B(a^{n+1})\quad\forall n\geq 0.    
	\end{equation}
	
	We recall some properties of these sequences. The first is Bregman's result \cite{Bregman65} on weak convergence.
	
	\begin{prop}\label{PropositionWeakConvergence}
		There exists $x\in A\cap B$ such that both $\{a^n\}$ and $\{b^n\}$ weakly converge to $x$.
	\end{prop}
	
	The second is a direct consequence of the fact that projection operators are non-expansive.
	
	\begin{prop}\label{PropositionFejer}
		The alternating projection sequences $\{a^n\}$ and $\{b^n\}$ are Fejér-monotone with respect to $A\cap B$, i.e. for every $x\in A\cap B$ the inequalities
		$$
		\|a^{n+1}-x\|\leq \|a^n-x\|\quad\text{ and }\quad\|b^{n+1}-x\|\leq \|b^n-x\|
		$$
		hold.
	\end{prop}
	\noindent
	Note that a Fejér-monotone sequence is always bounded.
	
	Now, let $X$ be a Hilbert lattice, $A\coloneqq (\spa\{v_1,\ldots,v_N\})^{\perp}$, where the vectors $v_i$ form an orthonormal system in $X$, and $B\coloneqq X^+$. Define $Q(x)\coloneqq \sum_{i=1}^N \langle x,v_i\rangle v_i$. We have $A\cap B\neq \emptyset$ since $0\in A\cap B$, and
	$$
	P_A(x) = x -Q(x),\quad P_B(y) = y^+\quad\text{ for }x,y\in X.
	$$
	We recall a result by Bausche and Borwein \cite{BauschkeBorwein93} which ensures the convergence of the alternating projection sequences when $A$ and $B$ as above satisfy some further conditions.
	
	\begin{prop}\label{PropositionConditionsBB}
		Let $X$ be a Hilbert lattice, $A\coloneqq (\spa\{v_1,\ldots,v_N\})^{\perp}$ as above and $B\coloneqq X^+$. The alternating projection sequences with respect to $A$ and $B$ converge in norm to a point in $A\cap B$ if one of the following is satisfied.
		\begin{itemize}
			\item[1)] $Q(X^+)\subseteq X^+$.
			\item[2)] $X^+\cap \spa\{v_1,\ldots,v_N\} = \{0\}$. 
		\end{itemize}
	\end{prop}
	
	\begin{rmk}
		The condition 1) is proved in \cite[Theorem 5.1]{BauschkeBorwein93} by showing that the sequence $\{b^n\}$ on the cone is decreasing with respect to the lattice structure, so that it converges thanks to Proposition \ref{PropositionDecreasing}. In particular, 1) holds if $\{v_1,\ldots,v_N\}$ is an orthogonal subset of $X^+\cup (-X^+)$.
		
		The condition 2) is instead proved by showing that it implies the regularity of the couple $(A,B)$ (see \cite[Theorem 5.3]{BauschkeBorwein93}). 
	\end{rmk}
	
	\begin{rmk}\label{RemarkCodimension1}
		When $N=1$, the norm convergence of the sequences is guaranteed because either $v_1\in X^+ \cap (-X^+)$ or $X^+\cap\; \spa\{v_1\}=\{0\}$. An analogue dichotomy for $\{v_1,\ldots,v_N\}$ fails if $N\geq 2$.
	\end{rmk}
	In the following sections, we will show that every new result applies to cases which cannot be solved by Proposition \ref{PropositionConditionsBB}. 
	
	\section{Main result: subspace generated by disjoint vectors}
	In this section we  consider our problem under the assumption that the orthonormal vectors $v_1,\ldots,v_N$ are disjoint with respect to the lattice structure induced by the cone $X^+$ in the Hilbert lattice $X$.
	
	We begin by assuming that $N=2$ and $X=\ell^2(\N)$: by Proposition \ref{PropositionDisjointEssSup}, two vectors $v_1$ and $v_2$ are disjoint if and only if $\supp(v_1)\cap\supp(v_2)=\emptyset$.
	
	\begin{thm}\label{TheoremDisjointEll2N}
		Let $X\coloneqq \ell^2(\N)$ and let $\{v_1,v_2\}$ be two disjoint vectors 
		in $X$. Let $A\coloneqq (\spa\{v_1,v_2\})^{\perp}$  and $B\coloneqq (\ell^2(\N))^+$. Then the alternating projection sequences with respect to $A$ and $B$ starting from a point $x\in X$ converge in norm to a point $y\in A\cap B$.
	\end{thm}
	
	Before proving this result we show that Theorem \ref{TheoremDisjointEll2N} applies to situations where the hypotheses of Proposition \ref{PropositionConditionsBB} are not satisfied.
	
	\begin{rmk}
		The couple of normalized vectors 
		\begin{align*}
			&v_1=\frac{\sqrt{6}}{\pi}\lp 1,0,\frac{1}{2},0,\frac{1}{3},0,\dots\rp,\\
			&v_2=\frac{\sqrt{6}}{\pi}\lp 0,1,0,-\frac{1}{2},0,\frac{1}{3},0,\dots\rp
		\end{align*}
		is such that the convergence of the corresponding alternating projection sequences follows from Theorem \ref{TheoremDisjointEll2N} but is a not a consequence of Bausche-Borwein's conditions in Proposition \ref{PropositionConditionsBB}, since $v_1\in X^+$ and $Q(e_4) \notin X^+ $, where $e_n$ denotes the $n$-th element of the standard basis of $\ell_2$.
	\end{rmk}
	
	\begin{proof} Up to normalization, we can assume that the vectors $v_1$ and $v_2$ have norm 1. 		
		Let $S_1\coloneqq\supp(v_1)$, $S_2\coloneqq\supp(v_2)$ and  $Z\coloneqq \N\setminus (S_1\cup S_2)$. Every element $x\in X$ can be decomposed as a sum $x=x_{S_1} + x_{S_2}+x_Z$ where $\supp(x_T) \subseteq T$ for every $T\in\{S_1,S_2,Z\}$. By hypothesis, one has $(v_1)_{S_2} = (v_1)_Z = 0$ and $(v_2)_{S_1} = (v_2)_Z = 0$.
		
		Let $\{a^n\}$ and $\{b^n\}$ be the alternating projection sequences on $A$ and $B$ respectively, starting from a point $x$. In particular, one has 
		\begin{align*}
			a^{n+1}=P_A(b^n) = b^n -\langle b^n, v_1\rangle v_1 -\langle b^n, v_2\rangle v_2   
		\end{align*}
		and the decomposition in disjoint parts $b^n = b^n_{S_1} + b^n_{S_2} + b^n_Z$ implies $\langle b^n, v_1 \rangle = \langle b^n_{S_1}, v_1 \rangle$ and $\langle b^n, v_2 \rangle = \langle b^n_{S_2}, v_2 \rangle$, so that
		\begin{align*}
			a^{n+1} &= b^n -\langle b^n_{S_1}, v_1\rangle v_1 -\langle b^n_{S_2}, v_2\rangle v_2\\
			&=(b^n_{S_1} -\langle b^n_{S_1}, v_1\rangle v_1) + (b^n_{S_2} -\langle b^n_{S_2}, v_2\rangle v_2)+b^n_Z
		\end{align*}
		and the last three terms coincide with $a^{n+1}_{S_1}, a^{n+1}_{S_2}$ and $a^{n+1}_Z$ respectively.
		
		Now, the computation of the positive part in $\ell^2(\N)$ respects the disjunction of supports: in particular one has
		\begin{align*}
			b^{n+1}=P_B(a^{n+1})=(a^{n+1})^+ &= (a^{n+1}_{S_1}+a^{n+1}_{S_2}+a^{n+1}_Z)^+ = (a^{n+1}_{S_1})^+ + (a^{n+1}_{S_2})^+ + (a^{n+1}_Z)^+\\
			&=(b^n_{S_1} -\langle b^n_{S_1}, v_1\rangle v_1)^+ + (b^n_{S_2} -\langle b^n_{S_2}, v_2\rangle v_2)^+ + (b^n_Z)^+.
		\end{align*}
		In this way we proved that the alternating projection sequences $\{a^n,b^n\}$ can be decomposed in three distinct and pairwise disjoint alternating sequences $\{a^n_{S_1}, b^n_{S_1}\}$, $\{a^n_{S_2}, b^n_{S_2}\}$ and $\{a^n_Z, b^n_Z\}$, and the convergence is proved for each one separately. In fact, the sequence with support in $Z$ is constant, i.e. $a^{n+1}_Z=b^{n+1}_Z=b^n_Z$ for every $n$; the two remaining couples of sequences instead converge because they arise as a moment problem of codimension 1 with starting point $x_{S_1}$ and $x_{S_2}$ respectively, where the corresponding linear subspaces are $\spa\{v_1\}$ and $\spa\{v_2\}$, and the convergence is guaranteed by Remark \ref{RemarkCodimension1}. This implies the convergence of our initial couple of alternating sequences.
	\end{proof}
	The proof of Theorem \ref{TheoremDisjointEll2N} can be easily adapted for $N$ disjoint vectors. Hence, we obtain the following result.
	\begin{coro}
		Let $X\coloneqq \ell^2(\N)$ and let $\{v_1,v_2,\ldots,v_N\}$ be a collection of pairwise disjoint vectors in $X$. Let $A\coloneqq (\spa\{v_1,\ldots,v_N\})^{\perp}$  and $B\coloneqq (\ell^2(\N))^+$. Then the alternating projection sequences with respect to $A$ and $B$ starting from a point $x\in X$ converge in norm to a point $y\in A\cap B$.
	\end{coro}
	\noindent
	\begin{rmk}
		Let $\alpha\in (-1,1)$, $\alpha\neq 0$ and consider the orthonormalized vectors of $\ell^2(\N)$
		\begin{align*}
			&v_1 = \frac{\sqrt{1-\alpha^2}}{|\alpha|}\lp \alpha,\alpha^2,\ldots,\alpha^n,\ldots\rp,\\
			&v_2 = \frac{\sqrt{1-\alpha^2}}{|\alpha|}\lp \alpha^2,-\alpha,\alpha^4,-\alpha^3,\ldots,\alpha^{2n},-\alpha^{2n-1},\ldots\rp.
		\end{align*}
		These two vectors do not satisfy the conditions of Theorem \ref{TheoremDisjointEll2N}; nonetheless, the linear combinations
		\begin{align*}
			&w_1\coloneqq v_1+\alpha v_2 = \frac{\sqrt{1-\alpha^2}}{|\alpha|}\lp \alpha+\alpha^3,0,\alpha^3+\alpha^5,0,\ldots,\alpha^{2n+1}+\alpha^{2n+3},0,\ldots\rp,\\
			&w_2\coloneqq -\alpha v_1 + v_2 = \frac{\sqrt{1-\alpha^2}}{|\alpha|}\lp 0,-\alpha^3-\alpha,0,-\alpha^5-\alpha^3,\ldots,0,-\alpha^{2n+3}-\alpha^{2n+1},\ldots\rp
		\end{align*}
		are disjoint vectors, and since $\spa\{v_1,v_2\}=\spa\{w_1,w_2\}$, the alternating projection sequences for $A=(\spa\{v_1,v_2\})^{\perp}$ and $B=X^+$ always converge by Theorem \ref{TheoremDisjointEll2N}.  We note, however, that the convergence was already assured since $\{w_1,w_2\}\subset B\cup (-B)$, and so Proposition \ref{PropositionConditionsBB} can be applied.
	\end{rmk}
	\begin{rmk}
		Similarly to the previous example, one can prove that alternating projection sequences converge when $\{v_1,v_2\}$ is an orthonormal system formed by
		\begin{align*}
			&v_1=\lp\alpha_1,\alpha_2, c_1\alpha_1, c_1\alpha_2, c_2\alpha_1, c_2\alpha_2,\ldots,c_n\alpha_1,c_n\alpha_2,\ldots\rp,\\
			&v_2=\lp\alpha_2,-\alpha_1, c_1\alpha_2, -c_1\alpha_1, c_2\alpha_2, -c_2\alpha_1,\ldots,c_n\alpha_2,-c_n\alpha_1,\ldots\rp
		\end{align*}
		with $(\alpha_1,\alpha_2)\neq (0,0)$ and $(c_i)_{i\in\N}\in\ell^2(\N)$. In fact, $\spa\{v_1,v_2\}=\spa\{w_1,w_2\}$ where $w_1\coloneqq \alpha_1v_1+\alpha_2 v_2$ and $w_2\coloneqq \alpha_2 v_1 -\alpha_1 v_2$ are disjoint vectors. The fact that this procedure works for any choice of $(c_i)_{i\in\N}\in\ell^2(\N)$ suggests that the possible convergence of the sequences for the moment problem does not depend on the speed of convergence of the components of the vectors $v_1$ and $v_2$.
	\end{rmk}
	
	A result analogous to 
	Theorem \ref{TheoremDisjointEll2N} can be proved for more general Hilbert lattices: we begin by showing it for a $L^2$ space.
	\begin{thm}\label{theoremDisjointL2Y}
		Let $(Y,\mu)$ be a measure space and $X=L^2(Y,\mu)$. Let  $\{f_1,\ldots,f_N\}$ be a collection of pairwise disjoint elements in $X$ and define $A\coloneqq (\spa\{ f_1,\ldots,f_N\})^{\perp}$ and $B\coloneqq X^+$.
		
		Then the alternating projection sequences with respect to $A$ and $B$ starting from a point $g\in X$ converge in norm to a point $h\in A\cap B$.
	\end{thm}
	\begin{proof}
		Without loss of generality, we can normalize so that the functions $f_i$ have norm 1. We prove the claim for $N=2$: the proof for generic $N$ follows by immediate adaptation.
		
		Let $g\in X$ and let $\{a^n\}$ and $\{b^n\}$ be the alternating projection sequences with respect to $A$ and $B$ starting from $g$. Let $E_1\coloneqq \esup (f_1)$, $E_2\coloneqq \esup (f_2)$ and $Z\coloneqq Y\setminus (E_1\cup E_2)$: by Proposition \ref{PropositionDisjointEssSup} one has $E_1\cap E_2=\emptyset$, so that $Y=E_1\cup E_2 \cup Z$ is a disjoint union. Moreover, every function $g\in L^2(Y,\mu)$ can be decomposed as $g= g_{E_1}+g_{E_2}+g_Z$, where $g_S\coloneqq (g\cdot \chi_S)(x)$ with $\chi_S$ being the characteristic function over $S$ and $S\in\{E_1,E_2,Z\}$. Note that $\langle g_{E_2}, f_1\rangle = \langle g_Z, f_1\rangle = \langle g_{E_1}, f_2\rangle = \langle g_Z, f_2\rangle = 0$ since disjoint elements are orthogonal.
		
		Thanks to this, for every function $b\in B$ the projection $P_A(b)$ on the linear subspace $A$ can be decomposed as $(P_A(b))_{E_1} + (P_A(b))_{E_2} + (P_A(b))_Z$: at the same time, one has
		\begin{align*}
			P_A(b) &= b -\langle b,f_1\rangle f_1 -\langle b,f_2\rangle f_2\\ 
			&= b_{E_1}+b_{E_2}+b_Z -\langle b_{E_1}+b_{E_2}+b_Z,f_1\rangle f_1 -\langle b_{E_1}+b_{E_2}+b_Z,f_2\rangle f_2\\
			&=(b_{E_1}-\langle b_{E_1},f_1\rangle f_1) + (b_{E_2}-\langle b_{E_2},f_2\rangle f_2) + b_Z
		\end{align*}
		and so one has almost everywhere the equalities $(P_A(b))_{E_1} = (b_{E_1}-\langle b_{E_1},f_1\rangle f_1)$, $(P_A(b))_{E_2} = (b_{E_2}-\langle b_{E_2},f_2\rangle f_2) $ and $(P_A(b))_Z = b_Z$, so that the corresponding equivalence classes are equal.
		
		Similarly, if $a\in A$, we can decompose $P_B(a)$ as $(P_B(a))_{E_1} + (P_B(a))_{E_2} + (P_B(a))_Z$: but the projection on $B=(L^2(Y,\mu))^+$ corresponds to taking the equivalence class of the positive part of the function, and almost everywhere one has the equalities
		\begin{align*}
			P_B(a) = a^+ = (a_{E_1}+a_{E_2}+a_Z)^+ = (a_{E_1})^+ + (a_{E_2})^+ + (a_Z)^+.    
		\end{align*}
		Since the essential supports of these three components are contained in $E_1, E_2, Z$ respectively, it follows that the equalities $(P_B(a))_{E_1} = (a_{E_1})^+$, $(P_B(a))_{E_2} = (a_{E_2})^+$ and $(P_B(a))_Z = (a_Z)^+$ hold almost everywhere, so that one has the equality for the corresponding equivalence classes.
		
		These facts together show that the convergence of the alternating sequences $\{a^n, b^n\}$ follows from the convergence of $\{a^n_{E_1},b^n_{E_1}\}, \{a^n_{E_2},b^n_{E_2}\}$ and $\{a^n_Z,b^n_Z\}$: the first two couples converge because they are given by a moment problem of codimension 1 with respect to $\spa\{f_1\}$ and $\spa\{f_2\}$ respectively, while the third couple is a constant sequence.
		
		The proof for $N\geq 3$ can be immediately recovered from an adaption of the one above.
	\end{proof}
	
	Finally, we combine the results from Theorem \ref{theoremDisjointL2Y} and Proposition \ref{PropositionIsoLattice} to obtain the corresponding result in any Hilbert lattice.

	\begin{thm}
		Let $X$ be a Hilbert lattice. Let $\{x_1,\ldots,x_N\}\subset X$ be a collection of pairwise disjoint vectors, and let  $A\coloneqq(\spa\{x_1,\ldots,x_N\})^{\perp}$.
		
		Then the alternating projection sequences with respect to $A$ and $X^+$ always converge to a point in $A\cap X^+$.
	\end{thm}
	
	\begin{proof}
		By Proposition \ref{PropositionIsoLattice}, there exist a measure space $(Y,\mu)$ and an isometric isomorphism $\psi: X \to L^2(Y,\mu)$ such that the lattice structure is preserved: in particular, one has $\psi(X^+)=B\coloneqq (L^2(Y,\mu))^+$, $\psi(x\land_{X^+} y)=\psi(x)\land_{B}\psi(y)$  and $\psi(x\lor_{X^+} y) = \psi(x) \lor_B \psi(y)$. If we consider $f_i\coloneqq \psi(x_i)$ for $i=1,\ldots,N$, the lattice preservation and the isometry imply that the functions $\{f_1,\ldots,f_N\}$ form an orthogonal system made by disjoint functions in $L^2(Y,\mu)$.
		
		Let $x_0\in X$ and consider the alternating projection sequence $\{a^n, c^n\}$ with respect to $A$ and $X^+$; let $W\coloneqq \psi(A)$. Since $\psi$ is an isometry, this linear subspace coincides in $L^2(Y,\mu)$ with $(\spa\{f_1,\ldots,f_N\})^{\perp}$. Consider then the elements $y_0\coloneqq \psi(x_0)$ and $w^n\coloneqq \psi(a^n), b^n\coloneqq \psi(c^n)$ for every $n\in\N$: we prove that this is an alternating projection sequence starting from $y_0$ with respect to $W$ and $B$.
		
		In fact, being $\psi$ an isometry, one has
		$$
		P_{W}(\psi(x)) = \psi(P_A(x)),\quad P_B(\psi(x)) = \psi(P_{X^+}(x))\quad\forall x\in X.
		$$
		Since $a^1=P_A(x_0)$, one has $w^1=\psi(a^1)=\psi(P_A(x_0))=P_W(\psi(x_0))=P_W(y_0)$. Similarly one has
		\begin{align*}
			&b^{n+1} = \psi(c^{n+1}) = \psi(P_{X^+}(a^{n+1})) = P_B(\psi(a^{n+1})) = P_B(w^{n+1}),\\
			&w^{n+1} = \psi(a^{n+1}) = \psi(P_A(c^n)) = P_W(\psi(c^n)) = P_W(b^n)
		\end{align*}
		for every $n$. This proves that we have a moment problem of codimension $N$ in $L^2(Y,\mu)$: by Theorem \ref{theoremDisjointL2Y}, these alternating projection sequences converge to a point $u\in W\cap B$, and therefore the original sequences in $X$ must converge to the point $z\coloneqq \psi^{-1}(u)\in A\cap X^+$, proving our claim.
	\end{proof}

	\section{A step further in $\ell^2(\N)$: subspace generated by vectors with finite intersection of supports}
	In this last section we try to extend our approach to a more general case. Indeed, we study our problem under the assumption that the subspace involved is generated by vectors with the intersection of supports given by a finite set, instead of the empty set as in our main result. 
	Namely, we prove a result concerning the moment problem in the specific space $\ell^2(\N)$, with a linear subspace of codimension 2 and the cone $(\ell^2(\N))^+$ of  elements with non-negative coordinates.

	\begin{thm}\label{TheoremFiniteSupport}
		Let $X\coloneqq \ell^2(\N)$, $A\coloneqq (\spa\{v_1,v_2\})^{\perp}$ with $v_1$ and $v_2$ an orthogonal system and $B\coloneqq (\ell^2(\N))^+$. Assume $v_1\in B$, while $v_2$ has components which are definitively non-negative or non-positive. Moreover, assume $\supp(v_1)\cap\supp(v_2)$ is finite.
		
		Then the alternating projection sequences with respect to $A$ and $B$ starting from a point $x\in X$ converge in norm to a point $y\in A\cap B$.
	\end{thm}
	Before proving this Theorem, we point out that the assumptions of Theorem \ref{TheoremFiniteSupport} hold also in cases where Proposition \ref{PropositionConditionsBB} does not apply.
	\begin{rmk}
		Assume that the two vectors $v_1$ and $v_2$ are given by
		\begin{align*}
			&v_1\coloneqq \frac{1}{\sqrt{2}}\lp 1,1,0,0,0,\ldots\rp,\\
			&v_2\coloneqq \sqrt{\frac{3}{7}}\lp 1, -1, \frac{1}{2},\frac{1}{4},\frac{1}{8},\ldots,\frac{1}{2^n},\ldots\rp.
		\end{align*}
		Let us recall that $e_n$ denotes the $n$-th element of the standard basis of $\ell^2$. Then, $v_2\notin B \cup (-B)$, so that $Q(e_3)=\langle e_3,v_1\rangle v_1 + \langle e_3,v_2\rangle v_2 = \frac{1}{2}\sqrt{\frac{3}{7}}\cdot v_2\notin B$; moreover $v_1\in \spa\{v_1,v_2\}\cap (\ell^2(\N))^+$.
	\end{rmk}
	
	\begin{proof} As in the previous theorems, we can assume that $v_1$ and $v_2$ have norm 1, since the normalization preserves all the hypotheses of the theorem.
		
		Let $v_1=(v_{1,n})_{n\in\N}$ and let us assume that $v_2=(v_{2,n})_{n\in\N}$ has components definitively non-negative.
		Let $x\in\ell^2(\N)$ and let $\{a^j\}$ and $\{b^j\}$ be the alternating projections sequences on $A$ and $B$ starting from $x$: both sequences weakly converge to the same point $y\in A\cap B$ by Proposition \ref{PropositionWeakConvergence}. Up to applying $P_A$ and $P_B$, we can assume that $x=b^0\in B$: moreover we can decompose the sequences as 
		$$b^j = b^j_V + b^j_I+ b^j_Z + b^j_N + b^j_P \quad\text{ and }\quad a^j = a^j_V + a^j_I+ a^j_Z + a^j_N + a^j_P,$$
		where $\supp(b^j_S), \supp(a^j_S)\subseteq S$ for every $S\in\{V,I,Z,N,P\}$ and the sets are defined as follows:
		\begin{itemize}
			\item $V\coloneqq \supp(v_1)\setminus\supp(v_2)$;
			\item $I\coloneqq \supp(v_1)\cap\supp(v_2);$
			\item $Z\coloneqq \N\setminus (\supp(v_1)\cup\supp(v_2))$;
			\item $N\coloneqq \{n\in \supp(v_2)\setminus\supp(v_1)\colon v_{2,n} < 0 \}$;
			\item $P\coloneqq \{n\in \supp(v_2)\setminus\supp(v_1)\colon v_{2,n} > 0 \}$.
		\end{itemize}
		
		We recall that if the sequence $\{b^j\}$ is norm convergent, then $\{a^j\}$ converges too. Hence, it is sufficient to prove the convergence of $\{b^j\}$.
		
		We prove the norm convergence of $b^j$ by studying the convergence of each one of the five parts in which the sequence has been divided. In order to do this, let us remember that $a^{j+1}=P_A(b^j) = b^j - \langle b^j,v_1\rangle v_1 - \langle b^j,v_2\rangle v_2$ and $b^{j+1}=(a^{j+1})^+$ for every $j$. The last equality is true componentwise, i.e. $b^{j+1}_n=\lp a^{j+1}_n\rp^+$ for every $n$.
		
		By definition, the sequences $b^j_I$ and $b^j_N$ have support contained in $I$ and $N$ respectively for every $j$. But $I$ and $N$ are finite sets by hypothesis, therefore the sequences converge in norm thanks to Lemma \ref{LemmaWeakConvFiniteSupport}.
		
		For every $n\in Z$ we have 
		$$a^{j+1}_n = b^j_n-\langle b^j,v_1\rangle \underbrace{v_{1,n}}_{=0} - \langle b^j,v_2\rangle \underbrace{v_{2,n}}_{=0}  = b^j_n.$$ 
		This implies $b^{j+1}_n = b^j_n$  and so $b^j_Z$ is a constant sequence. 
		
		For every $n\in V$ we have
		$$
		a^{j+1}_n = b^j_n-\langle b^j,v_1\rangle v_{1,n} - \langle b^j,v_2\rangle \underbrace{v_{2,n}}_{=0}  = b^j_n - \langle b^j,v_1\rangle v_{1,n}
		$$
		and since $b^j$ and $v_1$ both belong to $B$, we have $a^{j+1}_n\leq b^j_n$ and in turn $b^{j+1}_n\leq b^j_n$. Thus the sequence satisfies $b^{j+1}_I\leq b^j_I$ for every $j$, where the inequality refers to the lattice structure induced by $B$: therefore, the sequence $b^j_I$ converges in norm thanks to Proposition \ref{PropositionDecreasing}.  
		
		Finally, we study the norm convergence of $b^j_P$. We note that for every $n\in P$ we have
		$$
		a^1_n = b^0_n - \langle b^0,v_2\rangle v_{2,n}
		$$
		and since $b^1_n = (a^1_n)^+$, it follows $b^1_n\leq b^0_n$ if $\langle b^0,v_2\rangle\geq 0$, while $b^1_n=a^1_n\geq b^0_n$ if $\langle b^0,v_2\rangle\leq 0$. In the latter case, one has
		$$
		a^2_n = b^1_n - \langle b^1,v_2\rangle v_{2,n} = b^0_n - \langle b^0+b^1,v_2\rangle v_{2,n}
		$$
		and the same alternative presents.
		
		We thus have two possibilities:\smallskip
		\begin{itemize}
			\item[1)] There exists $i\geq 0$ such that for every $\ell\geq 0$ the inequality  $\langle b^i+b^{i+1}+\cdots+b^{i+\ell},v_2\rangle\leq 0$ holds. Then, for every $\ell\geq 0$ one has 
			$$b^{i+\ell+1}_n = b^i_n - \langle b^i+b^{i+1}+\cdots+b^{i+\ell},v_2\rangle v_{2,n} = b^i_n - \sum_{k=0}^\ell\langle b^{i+k},v_2\rangle v_{2,n}\quad\forall\; n\in P.$$
			For every such value of $n$, the weak convergence of $b^j$ implies $b^j_n \to y_n $  as $j\to +\infty$ and so $y_n = b^i_n - \sum_{k=0}^{+\infty}\langle b^{i+k},v_2\rangle v_{2,n}\in\R$ (and in fact the series of scalar products is convergent). But then for every $j\geq i$ we have  
			\begin{align*}
				\sum_{n\in P}(b^j_n-y_n)^2 = \sum_{n\in P}\lp \sum_{k=j}^{+\infty} \langle b^k,v_2\rangle\rp^2 v_{2,n}^2 \leq \lp \sum_{k=j}^{+\infty} \langle b^k,v_2\rangle\rp^2\|v_2\|^2 \to 0\quad\text{ as }j\to +\infty
			\end{align*}
			and therefore $b^j_P$ converges in norm. Together with the convergence of the previous four parts, it implies the norm convergence of the sequence.\smallskip\\
			\item[2)] For every $j\geq 0$, there exists a value $k\geq 1$ such that 
			$$\langle b^j,v_2\rangle\leq 0,\; \langle b^j+b^{j+1},v_2\rangle\leq 0,\;\ldots,\langle b^j+b^{j+1}+\cdots+b^{j+k-1},v_2\rangle\leq 0$$
			and
			$$\langle b^j+\cdots+b^{j+k-1}+b^{j+k},v_2\rangle\geq 0.$$
			It follows that
			$$
			a^{j+k+1}_n = b^j_n - \langle b^j+\cdots+b^{j+k-1}+b^{j+k},v_2\rangle v_{2,n} \leq b^j_n\quad\forall n\in P
			$$
			and this produces $b^{j+k+1}_P\leq b^j_P$ with respect to the lattice structure.
			
			Starting from $b^0_P$, this procedure gives a decreasing subsequence $b^{j_n}_P$ which therefore converges in norm by Proposition \ref{PropositionDecreasing}. Together with the norm convergence of the four previous parts, this yields the norm convergence to $y \in A\cap B$ of the subsequence $b^{j_n}$. But Proposition \ref{PropositionFejer} yields
			$$
			\|b^{j+1}-y\|\leq \|b^j-y\|\quad\forall j\in\N
			$$
			and the norm convergence of $b^j$ follows from the one of the subsequence $b^{j_n}$.
		\end{itemize}
		This proves the claim when $v_{2,n}\geq 0$ definitively: the proof for $v_{2,n}\leq 0$ is completely similar.
	\end{proof}
	This result admits the following generalization whenever the codimension of $A$ is $N\geq 3$: the proof goes along the same steps as the one of Theorem \ref{TheoremFiniteSupport}.
	
	\begin{coro}
		Let $X\coloneqq \ell^2(\N)$ and, for $N\geq 3$, let $\{v_1,v_2,\ldots,v_N\}$ be an orthogonal family in $X$. Let $A\coloneqq (\spa\{v_1,\ldots,v_N\})^{\perp}$  and $B\coloneqq (\ell^2(\N))^+$. Assume $v_1,v_2,\ldots,v_{N-1}\in B$ while $v_N$ has components which are definitively non-negative or non-positive. Moreover, assume $\supp(v_i)\cap\supp(v_N)$ is finite for $1\leq i\leq N-1$.
		
		Then the alternating projection sequences with respect to $A$ and $B$ starting from a point $x\in X$ converge in norm to a point $y\in A\cap B$.
	\end{coro}
	
	\begin{rmk}
		The general study for the moment problem of a linear subspace is not complete yet, but in view of our results a possible way to attack it could be the following. First, assuming the intersection of supports is finite, one should try to remove the assumption about the vectors being mostly in the cone and then the one about the signs of the components being definitively constant (see Theorem \ref{TheoremFiniteSupport}). Once one achieves this result, the next step would be to overcome the assumption concerning the finiteness of the intersection of supports and then moving on to the setting of generic Hilbert lattices. Finally, one should extend this approach to the moment problem where the linear subspace is replaced by an affine subspace.
	\end{rmk}

\end{document}